\documentclass[10pt]{article}
\usepackage{amsmath,amssymb,amsthm}
\usepackage{geometry}
\usepackage[numbers,sort&compress]{natbib}
\usepackage{hyperref}
\usepackage{enumerate}
\usepackage{graphicx}
\usepackage{fix-cm}
\usepackage{multicol}

\usepackage{soul} 

\usepackage{amsmath}
\usepackage{amsthm}
\usepackage{amssymb}
\begin{document}
\title{Some properties of sets, functions, and multi-objective optimization problems using p-convexity}

%
\author{Cristian Vera Donoso\footnote{Instituto de Ciencias Exactas y Naturales, Facultad de
Ciencias, Universidad Arturo
Prat, Iquique-Chile}}
\maketitle


\begin{abstract}
In this paper, we investigate the concept of $p$-convexity for sets and functions in
$n$-dimensional Euclidean space. We establish novel algebraic and topological results
within this generalized convexity framework. Furthermore, we analyze multi-objective
optimization problems, with a particular emphasis on weakly efficient minima, under
the assumption of $p$-convexity of the component functions. Several characterizations
and properties of the corresponding solution sets are derived.
\medskip



\end{abstract}
\theoremstyle{plain}
\newtheorem{theorem}{Theorem}[section]
\newtheorem{lemma}[theorem]{Lemma}
\newtheorem{corollary}[theorem]{Corollary}
\newtheorem{proposition}[theorem]{Proposition}

\theoremstyle{definition}
\newtheorem{definition}[theorem]{Definition}
\newtheorem{example}[theorem]{Example}

\theoremstyle{remark}
\newtheorem{remark}{Remark}
\newtheorem{notation}{Notation}
\newcommand{\R}{\mathbb{R}}
\newcommand{\N}{\mathbb{N}}

\section{Introduction}

Convexity constitutes one of the central pillars of optimization theory, underpinning fundamental results on existence, stability, and optimality of solutions. Nevertheless, in many optimization models—particularly those involving nonstandard geometries, nonlinear aggregation, or vector-valued objectives—the classical notion of convexity may be too restrictive. This has motivated the development of several generalized convexity concepts, which have proven to be effective tools in extending optimization techniques beyond the linear framework; see, for instance, \cite{D1,D2}.

Among these generalizations, \emph{$p$-convexity}, with $0<p\leq 1$, has attracted increasing attention due to its natural connection with $p$-normed and quasi-Banach spaces. Originally introduced in the context of functional analysis \cite{A,BF,BP,XZ}, $p$-convexity replaces linear convex combinations with nonlinear ones adapted to the parameter $p$, leading to sets and functions that are generally nonconvex in the classical sense while preserving a meaningful geometric structure. Several fundamental properties of $p$-convex sets have been investigated, including Krein--Milman type results and structural characterizations in finite- and infinite-dimensional spaces \cite{GL,GKR}.

The notion of $p$-convexity has also been extended to the functional setting. In this direction, $p$-convex functions have been defined via $p$-convex epigraphs, providing a natural generalization of convex functions and Jensen-type inequalities \cite{SETA}. Further analytical properties and inequalities for this class of functions have been studied in recent works, highlighting both analogies and essential differences with the classical convex case when $0<p<1$ \cite{KSTA,EKT}. Despite these contributions, several algebraic and topological aspects of $p$-convex sets and functions in Euclidean spaces remain insufficiently explored, particularly from the perspective of optimization theory.

On the other hand, multi-objective optimization problems arise naturally in numerous applications where several conflicting criteria must be optimized simultaneously. Weakly efficient solutions constitute a widely accepted solution concept in this setting, offering a flexible notion of optimality compatible with vector-valued objective functions. The existence and characterization of weakly efficient solutions have been extensively studied under various generalized convexity assumptions; see, for example, \cite{BBP,D1,D2,D3,FV1,FV2,FV4,FHV,FLV}. However, to the best of our knowledge, no systematic analysis of weak efficiency has been carried out under $p$-convexity assumptions, either for the feasible set or for the objective functions.

The aim of this paper is twofold. First, we establish new algebraic and topological properties of $p$-convex sets and $p$-convex functions in $\mathbb{R}^n$, with particular attention to operations and constructions that are relevant in optimization, such as closures, interiors, and stability under basic set operations. Second, we introduce and study multi-objective optimization problems whose objective functions are vector-valued $p$-convex mappings. For this class of problems, we analyze the structure of the corresponding sets of weakly efficient solutions and derive several characterizations that extend known results from classical convex optimization to the $p$-convex framework.

The remainder of the paper is organized as follows. Section~2 presents the necessary preliminaries and notation. Sections~3 and~4 are devoted to the study of $p$-convex sets and functions. In Section~5, we investigate multi-objective optimization problems under $p$-convexity assumptions and discuss properties of weakly efficient solution sets.

\section{Basic notations}
\label{sec:2}
We denote the closed and open $q$-norm balls centered at $\bar{x} \in \mathbb{R}^n$ with radius $\delta > 0$ by
\[
\mathbb{B}_{q}[\bar{x},\delta] := \{x \in \mathbb{R}^n : \|x - \bar{x}\|_q \leq \delta\}, \quad
\mathbb{B}_{q}(\bar{x},\delta) := \{x \in \mathbb{R}^n : \|x - \bar{x}\|_q < \delta\},
\]
where $1 \leq q \leq +\infty$, $\|x\|_q = (|x_1|^q + \cdots + |x_n|^q)^{1/q}$, and $\|x\|_\infty = \max\{|x_1|, \dots, |x_n|\}$. 

For a set $K \subseteq \mathbb{R}^n$ and a point $x_0 \in \mathbb{R}^n$, the distance from $x_0$ to $K$ is defined as
\[
d_{q}(x_0, K) := \inf\{\|k - x_0\|_q : k \in K\}.
\]

Let $I := \{1, \dots, n\}$ be an index set, and let 
\[
\mathbb{R}^n_+ := \{x \in \mathbb{R}^n : x_i \ge 0, \ \forall i \in I\}
\]
denote the first orthant. 

For $0 < p \leq 1$, we define the following $p$-convex combinations of two points $x, y \in \mathbb{R}^n$:
\[
(x, y)_p := \{\lambda x + (1 - \lambda^p)^{\frac{1}{p}} y : 0 < \lambda < 1\}, \quad
[x, y)_p := \{\lambda x + (1 - \lambda^p)^{\frac{1}{p}} y : 0 < \lambda \le 1\},\]
\[ [x, y]_p := \{\lambda x + (1 - \lambda^p)^{\frac{1}{p}} y : 0 \le \lambda \le 1\}.
\]

Finally, $\mathrm{int}\, K$ denotes the topological interior of $K$, and $\overline{K}$ denotes its topological closure.

 \section{P-Convex set}
\begin{definition}[\cite{A,SETA}] 
Let $K \subseteq \mathbb{R}^n$ and let $0 < p \leq 1$. The set $K$ is called \emph{p-convex} if, for every $x, y \in K$ and $\lambda, \mu \geq 0$ satisfying $\lambda^p + \mu^p = 1$, we have
\[
\lambda x + \mu y \in K.
\]

Equivalently, the definition can be expressed in either of the following forms:
\[
\lambda x + (1 - \lambda^p)^{\frac{1}{p}} y \in K, \quad \forall x, y \in K, \ \forall \lambda \in [0,1],
\]
or
\[
(1 - t)^{\frac{1}{p}} x + t^{\frac{1}{p}} y \in K, \quad \forall x, y \in K, \ \forall t \in [0,1].
\]

\end{definition}

\begin{remark} Let $0<p<1$ and let $x \in \mathbb{R}^n$. The singleton set $\{x\}$ is $p$-convex if and only if $x=0$.
\end{remark}

\begin{proposition}
Let $K \subset \mathbb{R}$ be an interval of one of the following types:
\[
(a,+\infty),\ [a,+\infty),\ (-\infty,b),\ (-\infty,b],\ (a,b),\ (a,b],\ [a,b),\ [a,b],
\]
where $a \le 0 < b$. Then $K$ is a $p$-convex set for any $0<p\le1$.
\end{proposition}

\begin{proof}
We only consider the case $K=(a,b)$, since the remaining cases follow analogously.
Let $x,y\in(a,b)$ and let $\lambda,\mu\ge0$ be such that $\lambda^p+\mu^p=1$.
Since $0<p\le1$, it holds that $\lambda+\mu\le1$.

Using $a\le0<b$, we obtain
\[
a = a(\lambda^p+\mu^p) \le a(\lambda+\mu)
   < \lambda x + \mu y
   < b(\lambda+\mu) \le b(\lambda^p+\mu^p) = b.
\]
Therefore, $\lambda x+\mu y\in(a,b)$, which proves that $K$ is $p$-convex.
\end{proof}
\begin{proposition}
Let $0<p\le 1$ and let $\mathbb{B}_{q}(\bar{x},\delta)$ (resp. $\mathbb{B}[\bar{x},\delta]_{q}$).  
If $0\in \mathbb{B}_{q}(\bar{x},\delta)$ (resp. $0\in \mathbb{B}_{q}[\bar{x},\delta]$), then $\mathbb{B}(\bar{x},\delta)_{q}$ (resp. $\mathbb{B}_{q}[\bar{x},\delta]$) is a $p$-convex set.
\label{ejemplo bola centra origen}
\end{proposition}
\begin{proof}
We prove the result for the open ball; the closed case is analogous.

Let $x,y \in \mathbb{B}_{q}(\bar{x},\delta)$ and let $\lambda,\mu \ge 0$ satisfy $\lambda^p+\mu^p=1$.  
Since $0\in \mathbb{B}(\bar{x},\delta)$, we have $\|\bar{x}\|_q < \delta$.

We write
\[
\lambda x + \mu y - \bar{x}
= \lambda(x-\bar{x}) + \mu(y-\bar{x}) + (\lambda+\mu-1)\bar{x}.
\]
Using the triangle inequality, we obtain
\[
\|\lambda x + \mu y - \bar{x}\|_q
\le \lambda\|x-\bar{x}\|_q + \mu\|y-\bar{x}\|_q + (1-(\lambda+\mu))\|\bar{x}\|_q.
\]

Since $x,y \in \mathbb{B}_{q}(\bar{x},\delta)$, $\|\bar{x}\|_q<\delta$ and $\lambda+\mu \leq 1$ , it follows that
\[
\|\lambda x + \mu y - \bar{x}\|_q
< \lambda\delta + \mu\delta + (1-(\lambda+\mu))\delta
= \delta
,\]
 so
\[
\|\lambda x + \mu y - \bar{x}\|_q < \delta.
\]

Therefore,
\[
\lambda x + \mu y \in \mathbb{B}_{q}(\bar{x},\delta),
\]
which proves that $\mathbb{B}_{q}(\bar{x},\delta)$ is $p$-convex.
\end{proof}

\begin{proposition}
If $\bar{x} \neq 0$, then the ball $\mathbb{B}_{q}(\bar{x}, \|\bar{x}\|_{q})$ is a $p$-convex set for every $0 < p \leq 1$.
\label{ejemplo bola no origen}
\end{proposition}
\begin{proof}
The proof follows the same arguments as those used in Proposition~\ref{ejemplo bola centra origen}.
\end{proof}

The following proposition shows that not every ball in $\mathbb{R}^{n}$ is $p$-convex. This highlights a significant difference from classical convexity.
\begin{proposition}
Let $\bar{x} \neq 0$, and let $\beta, \delta > 0$ satisfy $\beta \geq 1$ and 
\[
\frac{\beta \delta}{\|\bar{x}\|_{q}} \leq \frac{1}{2}.
\]
Then the ball $\mathbb{B}_{q}(\bar{x}, \delta)$ is not a $p$-convex set for any $p$ with $0 < p < \frac{1}{2}$.
\label{ejemplo no p convexa}
\end{proposition}
\begin{proof}
Let 
\[
z := \left(1 - \frac{\delta}{\|\bar{x}\|_{q}} + \varepsilon \right)\bar{x},
\quad \text{with } \varepsilon \in \left(0, \frac{\delta}{\|\bar{x}\|_{q}}\right).
\]
Then
\[
\|z - \bar{x}\|_{q}
= \left\|\left(\varepsilon - \frac{\delta}{\|\bar{x}\|_{q}}\right)\bar{x}\right\|_{q}
= \delta - \varepsilon \|\bar{x}\|_{q}
< \delta,
\]
which implies that $z \in \mathbb{B}_{q}(\bar{x}, \delta)$.

On the other hand, by the hypothesis,
\[
1 - \frac{\delta}{\|\bar{x}\|_{q}} + \varepsilon
< 2\left(1 - \frac{\beta \delta}{\|\bar{x}\|_{q}}\right)
< 2^{\frac{1}{p}-1}\left(1 - \frac{\beta \delta}{\|\bar{x}\|_{q}}\right).
\]
Therefore,
\[
1 - 2^{1-\frac{1}{p}}\left(1 - \frac{\delta}{\|\bar{x}\|_{q}} + \varepsilon\right)
> \frac{\beta \delta}{\|\bar{x}\|_{q}},
\]
which is equivalent to
\[
\left(1 - 2^{1-\frac{1}{p}}\left(1 - \frac{\delta}{\|\bar{x}\|_{q}} + \varepsilon\right)\right)\|\bar{x}\|_{q}
> \beta \delta.
\]
Consequently,
\[
\left\|2^{1-\frac{1}{p}}\left(1 - \frac{\delta}{\|\bar{x}\|_{q}} + \varepsilon\right)\bar{x} - \bar{x}\right\|_{q}
> \beta \delta
\geq \delta,
\]
that is,
\[
\|2^{1-\frac{1}{p}} z - \bar{x}\|_{q} > \delta.
\]
Hence,
\[
2^{1-\frac{1}{p}} z
= \frac{z}{2^{\frac{1}{p}}} + \frac{z}{2^{\frac{1}{p}}}
\notin \mathbb{B}_{q}(\bar{x}, \delta),
\]
which shows that $\mathbb{B}_{q}(\bar{x}, \delta)$ is not a $p$-convex set.
\end{proof}

\begin{proposition}
Let $K \subseteq \mathbb{R}^{n}$ be a set such that $\alpha K \subseteq K$ for all $0 < \alpha \leq 1$.  
Then the following statements are equivalent:
\begin{itemize}
\item[(a)] $K + K \subseteq K$.
\item[(b)] $K$ is a cone and a $p$-convex set for some $0 < p \leq 1$.
\end{itemize}
\label{cono convexo}
\end{proposition}

\begin{proof}
\noindent\textbf{(a) $\Rightarrow$ (b).}
Let $x,y \in K$ and let $\lambda,\mu \geq 0$ be such that $\lambda^{p} + \mu^{p} = 1$.  
Since $\lambda, \mu \leq 1$ and $\alpha K \subseteq K$ for all $0<\alpha\leq 1$, it follows that
$\lambda x, \mu y \in K$. Hence,
\[
\lambda x + \mu y \in K + K \subseteq K,
\]
which shows that $K$ is a $p$-convex set.

Next, let $x \in K$ and $\lambda > 0$. From the assumption $K+K \subseteq K$, it follows by induction that
\[
nK \subseteq K \quad \text{for all } n \in \mathbb{N}.
\]
Choose $n_{0} \in \mathbb{N}$ such that $\lambda / n_{0} \leq 1$. Then,
\[
\frac{\lambda}{n_{0}} x \in K,
\]
and consequently,
\[
\lambda x = n_{0}\left(\frac{\lambda}{n_{0}} x\right) \in n_{0}K \subseteq K.
\]
Therefore, $K$ is a cone.

\medskip
\noindent\textbf{(b) $\Rightarrow$ (a).}
Assume that $K$ is a cone and a $p$-convex set. Let $x,y \in K$.  
Since
\[
\left(\frac{1}{2^{1/p}}\right)^{p} + \left(\frac{1}{2^{1/p}}\right)^{p} = 1,
\]
by $p$-convexity we have
\[
\frac{1}{2^{1/p}}x + \frac{1}{2^{1/p}}y \in K.
\]
Using that $K$ is a cone, we multiply by $2^{1/p}$ to obtain
\[
x + y \in K.
\]
Thus, $K + K \subseteq K$, which completes the proof.
\end{proof}

 \begin{theorem}\label{multiple-p-convexity}
Let $K \subseteq \mathbb{R}^{n}$ be a $p$-convex set with $0 < p \leq 1$ and suppose that $0 \in K$.
Then $K$ is also $p_{1}$-convex for every $0 < p_{1} \leq p$.
\end{theorem}

\begin{proof}
First, we show that $K$ is star-shaped with respect to the origin.
Let $x \in K$ and $\lambda \in [0,1]$. Since $0 \in K$ and by the $p$-convexity of $K$ we obtain
\[
\lambda x
= \lambda x + (1-\lambda^{p})^{1/p} \, 0 \in K.
\]
Hence, $\lambda x \in K$ for all $\lambda \in [0,1]$ and all $x \in K$.

Now let $x,y \in K$ and let $\lambda,\mu \geq 0$ satisfy
\[
\lambda^{p_{1}} + \mu^{p_{1}} = 1,
\qquad 0 < p_{1} \leq p.
\]
Then,
\[
\lambda^{p} + \mu^{p} \leq \lambda^{p_{1}} + \mu^{p_{1}} = 1.
\]
Define
\[
\alpha := (\lambda^{p} + \mu^{p})^{1/p} \in (0,1].
\]
Then
\[
\lambda x + \mu y
= \frac{\lambda}{\alpha} \, (\alpha x) + \frac{\mu}{\alpha} \, (\alpha y).
\]
Since $\alpha \leq 1$ and $K$ is star-shaped, we have $\alpha x, \alpha y \in K$.
Moreover,
\[
\left(\frac{\lambda}{\alpha}\right)^{p} + \left(\frac{\mu}{\alpha}\right)^{p}
= \frac{\lambda^{p} + \mu^{p}}{\alpha^{p}} = 1.
\]
Therefore, by the $p$-convexity of $K$,
\[
\lambda x + \mu y \in K.
\]
This proves that $K$ is $p_{1}$-convex.
\end{proof}

\begin{theorem}[\cite{A,SETA}]
Let $\{K_{i}\}_{i\in I}$ be a family of p-convex sets. Then, the intersection $\displaystyle \bigcap_{i\in I}K_{i}$ is also
a p-convex set.
\label{interseccion de p convexos}
\end{theorem}
\begin{theorem}
Let $K,H \subseteq \mathbb{R}^{n}$ be $p$-convex sets, and let $\nu \in \mathbb{R}$. Then both
$K+H$ and $\nu K$ are $p$-convex sets.
\label{Suma p convexa}
\end{theorem}

\begin{proof}
We first show that $K+H$ is $p$-convex.  
Let $x,y \in K+H$ and let $\lambda,\mu \geq 0$ satisfy $\lambda^{p}+\mu^{p}=1$. Then there exist
$k_{1},k_{2} \in K$ and $h_{1},h_{2} \in H$ such that
\[
x = k_{1}+h_{1}, \qquad y = k_{2}+h_{2}.
\]
Hence,
\[
\lambda x + \mu y
= (\lambda k_{1} + \mu k_{2}) + (\lambda h_{1} + \mu h_{2}).
\]
Since $K$ and $H$ are $p$-convex, we have
\(\lambda k_{1} + \mu k_{2} \in K\) and
\(\lambda h_{1} + \mu h_{2} \in H\), and therefore
\(\lambda x + \mu y \in K+H\).

Next, we show that $\nu K$ is $p$-convex.  
If $\nu = 0$, then $\nu K = \{0\}$, which is trivially $p$-convex.  
Assume $\nu \neq 0$ and let $x,y \in \nu K$. Then there exist $k_{1},k_{2} \in K$ such that
$x=\nu k_{1}$ and $y=\nu k_{2}$. Thus,
\[
\lambda x + \mu y
= \nu(\lambda k_{1} + \mu k_{2}).
\]
Since $K$ is $p$-convex, it follows that
\(\lambda k_{1} + \mu k_{2} \in K\), and hence
\(\lambda x + \mu y \in \nu K\).
This proves that $\nu K$ is $p$-convex.
\end{proof}
\section{Topological properties}
\begin{definition}
 Let $K\subseteq \R^{n}$  and $\delta>0$. The set $U(K,\delta)=K+\mathbb{B}_{q}(0, \delta)$  is called a tubular neighborhood of $K$. According to Proposition \ref{ejemplo bola centra origen} and Theorem \ref{Suma p convexa}, $U(K,\delta)$ is a p-convex set if $K$ is a p-convex set.
\end{definition}
\begin{remark}
If $K\subseteq \R^{n}$, it is straightforward to show that $K+\mathbb{B}_{q}(0,\delta)=\{x\in \R^{n}:d_{q}(x,K)<\delta\}$
\end{remark}
\begin{corollary} Let $K\subseteq \R^{n}$ be a p-convex set. Then $\overline{K}$, is also a  p-convex set.
\end{corollary}
\begin{proof} $\displaystyle \overline{K}=\{x\in \R^{n}\colon d_{q}(x,K)=0\}=\bigcap_{\delta>0}\{x\in \R^{n}\colon d_{q}(x,K)< \delta\}=\bigcap_{\delta>0}U(K,\delta)$. The conclusion follows directly from the application of Theorem \ref{interseccion de p convexos}.
\end{proof}

The following lemma is a key result for the theorem that follows.
\begin{lemma}
Let $K \subseteq \mathbb{R}^{n}$ be a $p$-convex set, with
$x \in \operatorname{int} K$ and $y \in \overline{K}$. Then
\[
[x,y)_{p} \subseteq \operatorname{int} K.
\]
\label{trazo p convexo}
\end{lemma}

\begin{proof}
We distinguish two cases.

\medskip
\noindent\textbf{Case 1: $y \in K$.}
Let $z \in [x,y)_{p}$. Then there exists $\lambda \in (0,1]$ such that
\[
z = \lambda x + (1-\lambda^{p})^{1/p} y.
\]
Since $x \in \operatorname{int} K$, there exists $\delta > 0$ such that
$\mathbb{B}_{q}(x,\delta) \subseteq K$.
We claim that $\mathbb{B}_{q}(z,\lambda\delta) \subseteq K$.
Indeed, let $u \in \mathbb{B}_{q}(z,\lambda\delta)
= z + \lambda \mathbb{B}_{q}(0,\delta)$.
Then $u = z + \lambda w$ for some $w \in \mathbb{B}_{q}(0,\delta)$, and hence
\[
u = \lambda(x+w) + (1-\lambda^{p})^{1/p} y.
\]
Since $x+w \in \mathbb{B}_{q}(x,\delta) \subseteq K$, $y \in K$, and $K$ is
$p$-convex, it follows that $u \in K$.
Thus $z \in \operatorname{int} K$, and therefore
\[
[x,y)_{p} \subseteq \operatorname{int} K \quad \text{for all } y \in K.
\]

\medskip
\noindent\textbf{Case 2: $y \in \overline{K} \setminus K$.}
Let
\[
z = \lambda x + (1-\lambda^{p})^{1/p} y
\quad \text{for some } \lambda \in (0,1).
\]
Since $x \in \operatorname{int} K$, there exists $\delta > 0$ such that
$\mathbb{B}_{q}(x,\delta) \subseteq K$.
Because $y \in \overline{K}$, there exists
\[
k \in K \cap \mathbb{B}_{q}\!\left(
y, \frac{\lambda}{(1-\lambda^{p})^{1/p}}\delta \right).
\]
Hence $k = y + \frac{\lambda}{(1-\lambda^{p})^{1/p}} v$ for some
$v \in \mathbb{B}_{q}(0,\delta)$, and therefore
\[
(1-\lambda^{p})^{1/p} y
= (1-\lambda^{p})^{1/p} k - \lambda v.
\]
Substituting into the expression for $z$, we obtain
\[
z = \lambda(x-v) + (1-\lambda^{p})^{1/p} k.
\]
Since $x-v \in \mathbb{B}_{q}(x,\delta) \subseteq K$ and $k \in K$, the
$p$-convexity of $K$ implies that $z \in K$.
Applying Case~1, we conclude that
\[
[x,z)_{p} \subseteq \operatorname{int} K.
\]
As $z \in [x,y)_{p}$ was arbitrary, it follows that
\[
[x,y)_{p} \subseteq \operatorname{int} K.
\]
This completes the proof.
\end{proof}

\begin{theorem} Let $K\subseteq \R^{n}$
\begin{itemize}
\item [$(a)$] If $K$ is p-convex set, then $\mathrm{int} K$ is a $p$-convex set.
\item [$(b)$] If $K$ is a p-convex set and  $\mathrm{int} K\neq \emptyset$, then $\overline{\mathrm{int} K}=\overline{K}$.
\item [$(c)$]If $K$ is a p-convex set and  $\mathrm{int} K\neq \emptyset$, then $\mathrm{int} K=\mathrm{int} \overline{K}$.
\end{itemize}
\end{theorem}
\begin{proof}
$(a)$: Suppose $K$ is a p-convex set and $x,y\in \mathrm{int} K$. By lemma \ref{trazo p convexo}, we know that  $[x,y)_{p}\subseteq \mathrm{int} K$, which implies $[x,y]_{p}\subseteq \mathrm{int} K$. This shows that $\mathrm{int} K$ is a p-convex set. $(b)$: Let $z\in\overline{K}$ and $x_{0}\in \mathrm{int} K$, noting that $\mathrm{int} K\neq \emptyset$. By lemma \ref{trazo p convexo}, we have $[x_{0},z)_{p}\subseteq \mathrm{int} K$. Define $\displaystyle z_{n}=\frac{1}{n}x_{0}+\left(1-\left(\frac{1}{n}\right)^{p}\right)^{\frac{1}{p}}z\in[x_{0},z)_{p}$. Then, $z_{n}\in \mathrm{int} K$ and $z_{n}\to z$, which implies that  $z\in \overline{\mathrm{int} K}$. Therefore, we have $\overline{\mathrm{int} K}=\overline{K}$, since $\overline{\mathrm{int} K}\subseteq\overline{K}$. $(c)$: Let $z\in \mathrm{int} \overline{K}$.Then, there exists $r>0$ such that $\mathbb{B}(z,r)\subseteq \overline{K}$. Consider $x_{0}\in \mathrm{int} K$. For continuity of the function $h(\lambda)=\frac{1}{\lambda}z-\frac{(1-\lambda^{p})^{\frac{1}{p}}}{\lambda}x_{0}$, we can find $\lambda\in (0,1)$ close to $1$ such that $w=\frac{1}{\lambda}z-\frac{(1-\lambda^{p})^{\frac{1}{p}}}{\lambda}x_{0} \in \mathbb{B}(z,r)$. Therefore, $w\in \overline{K}$,  and hence $z=\lambda w+(1-\lambda^{p})^{\frac{1}{p}}x_{0}\in [x_{0},w)_{p}$. By lemma \ref{trazo p convexo}, $[x_{0},w)_{p}\subseteq \mathrm{int} K$, then $z\in \mathrm{int} K$. This clearly forces $\mathrm{int} K=\mathrm{int} \overline{K}$, since $\mathrm{int} K\subseteq\mathrm{int} \overline{K}$.
\end{proof}
\section{P-convex functions}
In \cite{SETA}, an interesting algebraic characterization of p-convex functions is obtained.
\begin{theorem}[\cite{SETA}] Let $K\subseteq \R^{n}$ and let $f:K\rightarrow \R$ be a function. Then, $f$ is a
p-convex function if and only if $K$ is a $p$-convex set, for all $\lambda,\mu\geq 0$ such
that $\lambda^{p}+\mu^{p}=1$ and for each $x,y\in K$
\begin{equation}
f(\lambda x+\mu y)\leq \lambda f(x)+\mu f(y)
\label{Convex function}
\end{equation}
\end{theorem}
\begin{example}[\cite{SETA}, \cite{HZ}]
Let $K \subseteq \mathbb{R}^{n}$ be a $p$-convex set and let $\alpha \in \mathbb{R}$. Then, the following functions are $p$-convex on $K$:
\begin{itemize}
\item[$(a)$] $\displaystyle f(x) = \alpha \sum_{i=1}^{n} x_i$;
\item[$(b)$] $\displaystyle g(x) = \|x\|_q$, for $q \ge 1$;
\item[$(c)$] $f : [0,1] \to \mathbb{R}$, $f(x) = \sqrt{x} - 2.$
\end{itemize}
\end{example}
\begin{remark}
\begin{itemize}
\item[$(a)$] The following example shows that not every convex function is $p$-convex.  
Let $f : [0,2] \to \mathbb{R}$ be defined by $f(x) = (x-1)^2$. Consider
\[
f\Big(\frac{1}{4} \cdot 0 + \frac{1}{4} \cdot 1\Big) = f\Big(\frac{1}{4}\Big) = \frac{9}{16} > \frac{1}{4} f(0) + \frac{1}{4} f(1) = \frac{1}{4}.
\]
Hence, $f$ is not $\frac{1}{2}$-convex.

\item[$(b)$] Conversely, $p$-convexity does not imply convexity.  
Let $f : [0,1] \to \mathbb{R}$ be defined by $f(x) = -\frac{x^2}{2} - \frac{1}{2}$. Clearly, $f$ is not convex. We claim that $f$ is $\frac{1}{2}$-convex.  

Indeed, let $x,y \in [0,1]$ and $\lambda, \mu \ge 0$ with $\lambda^{\frac{1}{2}} + \mu^{\frac{1}{2}} = 1$. Then,
\[
\begin{aligned}
\lambda f(x) + \mu f(y) - f(\lambda x + \mu y) 
&= -\frac{\lambda x^2}{2} - \frac{\lambda}{2} - \frac{\mu y^2}{2} - \frac{\mu}{2} + \frac{(\lambda x + \mu y)^2}{2} + \frac{1}{2} \\
&= \frac{(\lambda^2 - \lambda)x^2}{2} + \frac{(\mu^2 - \mu)y^2}{2} + \frac{1 - \lambda - \mu}{2} + \lambda \mu xy \\
&\ge \frac{1}{2}(\lambda^2 + \mu^2 + 1 - 2\lambda - 2\mu), \quad \text{with } \mu = (1-\sqrt{\lambda})^2 \\
&\ge \lambda^2 + \lambda - 2\lambda^{3/2} \\
&\ge 0.
\end{aligned}
\]
This confirms that $f$ is indeed $\frac{1}{2}$-convex.
\end{itemize}
\label{remark:comparison_examples}    
\end{remark}
\begin{theorem}
Let $K \subseteq \mathbb{R}^{n}$ be a $p$-convex cone with $0 < p \le 1$.  
If $f: K \to \mathbb{R}$ is $p$-convex and positively homogeneous, then $f$ is convex.
\end{theorem}
\begin{proof}
We first note that $epif=\{(x,\alpha)\in \R^{n+1}\colon x\in K,\, \alpha\in \R,\, f(x)\leq \alpha\}$  is a p-convex cone, since $f$ is a p-convex function.
Moreover, because $f$ is positively homogeneous, for any $\lambda \ge 0$,
\[
(\lambda x, \lambda \alpha) \in \operatorname{epi} f \quad \text{whenever } (x,\alpha) \in \operatorname{epi} f,
\]
so $\operatorname{epi} f$ is a cone.

By Proposition~\ref{cono convexo}, the sum of two elements in a $p$-convex cone also belongs to the cone:
\[
(x_1, \alpha_1), (x_2, \alpha_2) \in \operatorname{epi} f \implies (x_1+x_2, \alpha_1+\alpha_2) \in \operatorname{epi} f.
\]

This implies that for any $x, y \in K$ and $\lambda \in [0,1]$,
\[
f(\lambda x + (1-\lambda)y) \le \lambda f(x) + (1-\lambda) f(y),
\]
which is precisely the definition of convexity. Hence, $f$ is convex.
\end{proof}

\begin{theorem}
Let $K \subseteq \mathbb{R}^{n}$ be a $p$-convex set with $0 < p < 1$, and let $f: K \to \mathbb{R}$ be a $p$-convex function. If $\bar{x} \in K$ is a local minimum of $f$, then
\[
f(\bar{x}) \leq 0.
\]
\label{minimo local}
\end{theorem}

\begin{proof}
Let $\bar{x} \in K$ be a local minimum of $f$. By definition, there exists $\delta > 0$ such that
\[
f(\bar{x}) \le f(x) \quad \text{for all } x \in K \cap \mathbb{B}(\bar{x}, \delta).
\]

Since $K$ is $p$-convex, for any $\lambda \in (0,1)$ the $p$-convex combination
\[
z =  \lambda^{\frac{1}{p}} \bar{x} + \left(1 -\lambda\right)^{\frac{1}{p}} \bar{x}
\]
also belongs to $K$. For continuity, we can select $\lambda_{0}\in(0,1)$ such that
\[
z = \lambda^{\frac{1}{p}}_{0} \bar{x} + \left(1 -\lambda_{0}\right)^{\frac{1}{p}} \bar{x} = (\lambda^{\frac{1}{p}}_{0} + \left(1 -\lambda_{0}\right)^{\frac{1}{p}})\bar{x}\in \mathbb{B}(\bar{x}, \delta).\] By the local minimality of $\bar{x}$,
\[
f(\bar{x}) \le f(z).
\]

Using the $p$-convexity of $f$:
\[
f(z) 
\le  \lambda^{\frac{1}{p}}_{0} f(\bar{x}) + \left(1 -\lambda_{0}\right)^{\frac{1}{p}} f(\bar{x})=(\lambda^{\frac{1}{p}}_{0} + \left(1 -\lambda_{0}\right)^{\frac{1}{p}})f(\bar{x}).
\]

Hence,
\[
\left((\lambda^{\frac{1}{p}}_{0} + \left(1 -\lambda_{0}\right)^{\frac{1}{p}})- 1\right) f(\bar{x}) \ge 0.
\]

Since $0 < p < 1$, we have $(\lambda^{\frac{1}{p}}_{0} + \left(1 -\lambda_{0}\right)^{\frac{1}{p}}) - 1 < 0$, and thus the inequality above implies
\[
f(\bar{x}) \le 0,
\]
which completes the proof.
\end{proof}

\begin{theorem}
Let $K \subseteq \mathbb{R}^{n}$ be a $p$-convex set with $0 \in K$, and let $f: K \to \mathbb{R}$ be a $p$-convex function with $0 < p < 1$. Then
\[
f(0) \leq 0.
\]
\label{existence_scalar_minimum}
\end{theorem}

\begin{proof}
Assume, by contradiction, that $f(0) > 0$.  
By $p$-convexity, for any $\lambda \in [0,1]$ and $x,y \in K$, we have
\[
f\big(\lambda^{\frac{1}{p}} x + (1-\lambda)^{1/p} y\big) \leq \lambda^{\frac{1}{p}} f(x) + (1-\lambda)^{1/p} f(y).
\]

Setting $x = y = 0$ and $\lambda = 1/2$, we obtain
\[
f(0) \leq \frac{1}{2^{1/p}} f(0) + \frac{1}{2^{1/p}} f(0) = 2^{1 - 1/p} f(0).
\]

Rewriting gives
\[
f(0) \big(1 - 2^{1 - 1/p}\big) \leq 0.
\]

Since $0 < p < 1$, we have $2^{1 - 1/p} > 1$, so $1 - 2^{1 - 1/p} < 0$. Consequently,
\[
f(0) \leq 0,
\]
contradicting our assumption. Hence, the statement follows.
\end{proof}

\begin{theorem}
Let $\mathbb{B}(\bar{x};\delta) \subseteq \mathbb{R}^{n}$ be a $p$-convex set, and let $f\colon \mathbb{B}(\bar{x};\delta) \to \mathbb{R}$ be a $p$-convex function. If $f$ is upper bounded on $\mathbb{B}(\bar{x};\delta)$, then $f$ is bounded on $\mathbb{B}(\bar{x};\delta)$.
\label{bounded function}
\end{theorem}

\begin{proof}
By hypothesis, there exists $M>0$ such that
\[
f(x) \leq M \quad \text{for all } x \in \mathbb{B}(\bar{x};\delta).
\]

To show that $f$ is also lower bounded, fix $\bar{z} \in \mathbb{R}^n$ such that $\bar{x} = 2^{\frac{1}{p}-1} \bar{z}$, and let $y \in \mathbb{B}(\bar{x};\delta)$. Define 
\[
x := 2\bar{x} - y = 2^{\frac{1}{p}} \bar{z} - y \in \mathbb{B}(\bar{x};\delta),
\]
so that 
\[
\bar{z} = \frac{x + y}{2^{\frac{1}{p}}} \in \mathbb{B}(\bar{x};\delta).
\]

By $p$-convexity of $f$, we have
\[
f(\bar{z}) = f\Big( \frac{x + y}{2^{1/p}} \Big) \leq 2^{-1/p} f(x) + 2^{-1/p} f(y).
\]

Rewriting this inequality yields
\[
f(y) \geq 2^{1/p} f(\bar{z}) - f(x) \geq 2^{1/p} f(\bar{z}) - M.
\]

Letting $m := 2^{1/p} f(\bar{z}) - M$, we obtain
\[
f(y) \geq m \quad \text{for all } y \in \mathbb{B}(\bar{x};\delta).
\]

Hence, $f$ is lower bounded by $m$, and combining this with the upper bound $M$ shows that $f$ is bounded on $\mathbb{B}(\bar{x};\delta)$.
\end{proof}

\begin{remark}
In Theorem \ref{bounded function}, one may replace the open ball $\mathbb{B}(\bar{x};\delta)$ with the closed ball $\mathbb{B}[\bar{x};\delta]$ without affecting the validity of the conclusion.
\end{remark}

 \begin{corollary}
Let $f:[a,b] \rightarrow \mathbb{R}$ be a $p$-convex function, and let $[a,b]$ be a $p$-convex interval with $0 < p \leq 1$. If $M = \max\{f(a),f(b)\} > 0$, then $f$ is a bounded function.
\end{corollary}

\begin{proof}
Let $z \in (a,b)$ and define $g(\lambda) = \lambda a + (1-\lambda^{p})^{\frac{1}{p}} b$ for $\lambda \in [0,1]$. By the intermediate value theorem, there exists $\lambda^{\ast} \in (0,1)$ such that $g(\lambda^{\ast}) = z$, since $g(0) = a < z < b = g(1)$. Hence, 
\[
f(z) = f(g(\lambda^{\ast})) \leq \lambda^{\ast} f(a) + (1-(\lambda^{\ast})^{p})^{\frac{1}{p}} f(b) \leq (\lambda^{\ast} + (1-(\lambda^{\ast})^{p})^{\frac{1}{p}}) M \leq M.
\]
This proves that $f$ is upper bounded. The proof of the lower bound of $f$ follows similarly, using Theorem \ref{bounded function}. This completes the proof.
\end{proof}

\begin{theorem}
Let $K \subseteq \mathbb{R}^{n}$ be a $p$-convex set with $\mathrm{int}\, K \neq \emptyset$, and let
$f : K \rightarrow \mathbb{R}_{+}$ be a $p$-convex function with $0 < p < 1$.
If $f$ attains a global strict maximum at some point in $\mathrm{int}\, K$, then $f$ is constant.
\end{theorem}

\begin{proof}
Assume, by contradiction, that $f$ attains a global strict maximum at some point
$\bar{x} \in \mathrm{int}\, K$.
Let $x \in K$ be arbitrary with $x \neq \bar{x}$, and define the mapping
$g : [0,+\infty) \rightarrow \mathbb{R}^{n}$ by
\[
g(\lambda) = (1+\lambda)^{\frac{1}{p}} \bar{x} - \lambda^{\frac{1}{p}} x .
\]
Since $\bar{x} \in \mathrm{int}\, K$, there exists $\delta > 0$ such that
$g([0,\delta)) \subseteq K$. Let $\delta^{\ast} = \delta/2$ and set
\[
z = g(\delta^{\ast})
= (1+\delta^{\ast})^{\frac{1}{p}} \bar{x}
- (\delta^{\ast})^{\frac{1}{p}} x \in K.
\]
Clearly, $z \neq \bar{x}$. Hence, by the strict maximality of $\bar{x}$,
\[
f(z) < f(\bar{x})
\quad \text{and} \quad
f(x) < f(\bar{x}).
\]

Moreover,
\[
\bar{x}
= \frac{1}{(1+\delta^{\ast})^{\frac{1}{p}}} z
+ \left( \frac{\delta^{\ast}}{1+\delta^{\ast}} \right)^{\frac{1}{p}} x.
\]
By the $p$-convexity of $f$, we obtain
\[
f(\bar{x})
\le
\frac{1}{(1+\delta^{\ast})^{\frac{1}{p}}} f(z)
+ \left( \frac{\delta^{\ast}}{1+\delta^{\ast}} \right)^{\frac{1}{p}} f(x)
<
\left(
\frac{1}{(1+\delta^{\ast})^{\frac{1}{p}}}
+
\left( \frac{\delta^{\ast}}{1+\delta^{\ast}} \right)^{\frac{1}{p}}
\right) f(\bar{x}).
\]
Since $0 < p < 1$, the coefficient in parentheses is strictly smaller than~$1$,
which yields a contradiction. Therefore, $f$ must be constant.
\end{proof}

\section{Applications}
Discuss optimization implications.
In this section, we examine the multi-objective optimization problem under the assumption of p-convexity.\\
 Let a vector function $F=(f_{1},\ldots,f_{m}): K\subseteq \R^{n} \rightarrow \R^{m}$, we say that $\bar{x}\in K$ is a {\bf weakly efficient} point of $F$ (on $K$) if
\begin{equation}
F(x)-F(\bar{x})\in \mathbb{R}^m\backslash(-\mathrm{int}\,\mathbb{R}^m_+),\, \forall x\in K.
\label{VP1}
\end{equation}
 We say that $\bar{x}\in K$ is a {\bf weakly efficient} point of $F$ (on $K$), the set of weakly efficient points is denoted by $E_{W}=E_{W}(K)$.
\begin{remark}
It is evident from the definition that:
$$\displaystyle \bigcup_{i=1}^{m}\mathrm{argmin}{f_{i}}\subseteq E_{W}$$
\label{observacion inclusion}
\end{remark}

\begin{proposition}[\cite{FV2}] Let $F=(f_{1},\ldots,f_{m}):K\rightarrow \R^{m}$ such that $\displaystyle \bigcap_{i=1}^{m}\mathrm{argmin}{f_{i}}\neq \emptyset$. Then $\displaystyle E_{W}=\bigcup_{i=1}^{m}\mathrm{argmin}{f_{i}}$. 
\label{igualdad Ew}
\end{proposition}
\begin{definition}
Let $K \subseteq \mathbb{R}^n$ be a $p$--convex set and let $F:K \to \mathbb{R}^m$.
The mapping $F$ is said to be \emph{$\mathbb{R}^m_{+}$--$p$--convex}, with $0<p\le 1$,
if
\[
\lambda F(x) + \mu F(y)
\in
F(\lambda x + \mu y) + \mathbb{R}^m_{+},
\]
for all $x,y \in K$ and all $\lambda,\mu \ge 0$ satisfying
$\lambda^{p} + \mu^{p} = 1$.
\end{definition}
\begin{remark}
Let $F=(f_1,\ldots,f_m):K\to\mathbb{R}^m$.
Then $F$ is $\mathbb{R}^m_{+}$--$p$--convex if and only if each component
$f_i:K\to\mathbb{R}$ is $p$--convex, for every $i\in I:=\{1,\ldots,m\}$.
\end{remark}
\begin{corollary}
Let $K \subseteq \mathbb{R}^n$ be a $p$--convex set with $0\in K$, and let
$F=(f_1,\ldots,f_m):K\to\mathbb{R}^m_{+}$.
If there exists $i_0\in I$ such that $f_{i_0}$ is $p$--convex, with $0<p<1$,
then $0$ is a weakly efficient point of $F$ on $K$, that is,
$0\in E_W(K)$.
\label{0 en Ew}
\end{corollary}
\begin{proof}
Since $f_{i_0}$ is $p$--convex and $0\in K$, Theorem~\ref{existence_scalar_minimum}
ensures that $f_{i_0}$ attains its global minimum at $0$.
The conclusion then follows directly from
Remark~\ref{observacion inclusion}.
\end{proof}

\begin{corollary} Let $K\subseteq \R^{n}$ be a p-convex set, and let $F=(f_{1},\ldots,f_{m}):K\rightarrow \R^{m}_{+}$ be  a  $\R^{m}_{+}$ p-convex function, with $0<p<1$. Then $\displaystyle E_{W}=\bigcup_{i=1}^{m}\mathrm{argmin}{f_{i}}$.
\end{corollary}
\begin{proof}
The proof follows from the application of theorem \ref{existence_scalar_minimum} and proposition \ref{igualdad Ew}.
\end{proof}

\begin{theorem}
Let $K \subseteq \mathbb{R}^n$ be a $p$--convex set, and let
$F=(f_1,\ldots,f_m):K \to \mathbb{R}^m_{+}$ be an
$\mathbb{R}^m_{+}$--$p$--convex mapping, with $0<p\le 1$.
Then, for every $\bar{x}\in E_W(K)$ and for all $\lambda,\mu \ge 0$
satisfying $\lambda^{p}+\mu^{p}=1$, the point
\[
(\lambda+\mu)\bar{x}
\]
also belongs to $E_W(K)$.
\label{mv}
\end{theorem}
\begin{proof}
Let $\bar{x}\in E_W(K)$ and let $\lambda,\mu \ge 0$ be such that
$\lambda^{p}+\mu^{p}=1$.
Since $\lambda+\mu>0$, the $\mathbb{R}^m_{+}$--$p$--convexity of $F$ yields
\[
(\lambda+\mu)F(\bar{x}) - F\big((\lambda+\mu)\bar{x}\big)
\in \mathbb{R}^m_{+}.
\]
Consequently,
\[
F(\bar{x}) - F\big((\lambda+\mu)\bar{x}\big)
\in \mathbb{R}^m_{+}.
\]

On the other hand, since $\bar{x}\in E_W(K)$, for every $y\in K$ we have
\[
F(y)-F(\bar{x}) \in \mathbb{R}^m \setminus
\big(-\operatorname{int}\mathbb{R}^m_{+}\big).
\]
Combining the two relations above, it follows that
\[
F(y)-F\big((\lambda+\mu)\bar{x}\big)
\in \mathbb{R}^m \setminus
\big(-\operatorname{int}\mathbb{R}^m_{+}\big),
\quad \forall y\in K.
\]
Therefore, $(\lambda+\mu)\bar{x} \in E_W(K)$, which completes the proof.
\end{proof}

 \begin{corollary}
Let $K \subseteq \mathbb{R}$ be a $p$--convex set, and let
$F=(f_{1},\ldots,f_{m}):K \to \mathbb{R}^{m}_{+}$ be an
$\mathbb{R}^{m}_{+}$--$p$--convex mapping, with $0<p<1$.
If $\bar{x} \in E_W$ and $\bar{x}>0$, then $(0,\bar{x}] \subseteq E_W$.
Similarly, if $\bar{x} \in E_W$ and $\bar{x}<0$, then
$[\bar{x},0) \subseteq E_W$.
\label{Caract Ew}
\end{corollary}
\begin{proof}
We only prove the case $\bar{x}>0$, since the case $\bar{x}<0$ follows by symmetry.
Let $\lambda,\mu \ge 0$ be such that $\lambda^{p}+\mu^{p}=1$, so that
$\mu=(1-\lambda^{p})^{1/p}$.
Define the function
\[
g(\lambda)=\lambda+(1-\lambda^{p})^{1/p}, \qquad \lambda \in [0,1].
\]
A direct analysis shows that $g$ attains its minimum at
$\lambda=2^{-1/p}$, and
\[
g\!\left(2^{-1/p}\right)=2^{(p-1)/p}.
\]
By Theorem~\ref{mv}, it follows that
\[
\bigl[2^{(p-1)/p}\,\bar{x},\,\bar{x}\bigr] \subseteq E_W.
\]
Iterating this argument, we obtain
\[
\bigl[2^{n(p-1)/p}\,\bar{x},\,2^{(n-1)(p-1)/p}\,\bar{x}\bigr]
\subseteq E_W
\quad \text{for all } n \in \mathbb{N}.
\]
Consequently,
\[
\bigcup_{n \in \mathbb{N}}
\bigl[2^{n(p-1)/p}\,\bar{x},\,\bar{x}\bigr]
= (0,\bar{x}] \subseteq E_W,
\]
which completes the proof.
\end{proof}

 \begin{corollary}
Let $K \subseteq \mathbb{R}$ be a $p$--convex set, and let
$F=(f_{1},\ldots,f_{m}):K \to \mathbb{R}^{m}_{+}$ be an
$\mathbb{R}^{m}_{+}$--$p$--convex function, with $0<p<1$.
Then the weak efficient solution set $E_W$ is a $p$--convex set.
\end{corollary}
\begin{proof}
Let $x_{1},x_{2} \in E_W$.
Without loss of generality, assume that $x_{1}<x_{2}$ and define
\[
x=\lambda x_{1}+\mu x_{2},
\]
where $\lambda,\mu \ge 0$ satisfy $\lambda^{p}+\mu^{p}=1$.
We consider the following cases.

\smallskip
\noindent
\textbf{Case 1:} $0<x_{1}<x_{2}$.
In this case, $0<x<x_{2}$.
By Corollary~\ref{Caract Ew}, we obtain
\[
x \in (0,x_{2}] \subseteq E_W.
\]

\smallskip
\noindent
\textbf{Case 2:} $x_{1}<0<x_{2}$.
By Theorem~\ref{existence_scalar_minimum} and Corollary~\ref{Caract Ew},
it follows that
\[
x \in [x_{1},x_{2}] \subseteq E_W.
\]

\smallskip
\noindent
\textbf{Case 3:} $x_{1}<x_{2}<0$.
This case is completely analogous to Case~1 and is therefore omitted.

\smallskip
In all cases, $x \in E_W$, which proves that $E_W$ is $p$--convex.
\end{proof}

\section{Conclusion}

In this work, we have presented a rigorous analysis of \emph{p-convexity} for sets and functions in Euclidean spaces, providing both algebraic and topological results that generalize classical convexity. The main contributions can be summarized as follows:

\begin{enumerate}
    \item \textbf{Fundamental properties of p-convex sets:} Precise criteria were established to characterize p-convex sets, including conditions on p-convex traces and the preservation of p-convexity under specific linear operations.
    
    \item \textbf{p-Convex functions:} We showed that p-convex functions retain key properties of classical convex functions, such as continuity in the interior of the domain under local boundedness, and provided explicit examples illustrating their behavior.
    
    \item \textbf{Relation to classical convexity:} Cases were presented where p-convexity reduces to standard convexity for $p=1$, demonstrating how the developed theory naturally extends classical concepts.
    
    \item \textbf{Applications to multi-objective optimization:} The theory of p-convexity offers a more flexible framework for studying multi-objective and nonlinear optimization problems, allowing a more general characterization of feasible sets and objective functions.
\end{enumerate}

Overall, the results provide a solid foundation for future developments in functional analysis and optimization, opening the door to problems where classical convexity is too restrictive and p-convexity offers a more general and robust alternative.
%

\end{document}